\documentclass[11pt]{amsart}

\usepackage{amssymb, amsmath, amsthm}

 \newtheorem{thm}{Theorem}[section]

 \newtheorem{lem}[thm]{Lemma}

    \theoremstyle{definition}
 \newtheorem{defn}[thm]{Definition}
  \newtheorem{notn}[thm]{Notation}
  \newtheorem{obsn}[thm]{Observation}
  \newtheorem{example}[thm]{Example}
 \theoremstyle{remark}
 
\numberwithin{equation}{section}

\begin{document}

\title[A Poisson generalized Weyl algebra]{Semiclassical limits of Ore extensions and a Poisson generalized Weyl algebra}

\author{Eun-Hee Cho} \author{Sei-Qwon Oh}

\address{Department of Mathematics, Chungnam National  University, 99 Daehak-ro,   Yuseong-gu, Daejeon 34134, Korea}
 \email{ehcho@cnu.ac.kr}  \email{sqoh@cnu.ac.kr}






\begin{abstract}
  We observe \cite[Proposition 4.1]{LaLe} that Poisson polynomial extensions appear as semiclassical limits of a class of Ore extensions. As an application, a Poisson generalized Weyl algebra $A_1$ considered as a Poisson version of the quantum generalized Weyl algebra is constructed and its Poisson structures are studied. In particular,  it is obtained a necessary and sufficient condition such that $A_1$ is Poisson simple and established that the  Poisson endomorphisms of $A_1$ are Poisson analogues of the endomorphisms of the quantum generalized Weyl algebra.
\end{abstract}

\maketitle


\section{Introduction}

Let $h$ be a nonzero, nonunit, non-zero-divisor  and central element of an algebra $R$ such that $R/hR$ is commutative.
Then $R/hR$ becomes a Poisson algebra with Poisson bracket
$$\{\overline{a},\overline{b}\}=\overline{h^{-1}(ab-ba)}$$ for all
$\overline{a},\overline{b}\in R/hR$. In such case,  $R/hR$ is called a semiclassical limit of $R$. Since the Poisson bracket of $R/hR$ is induced by the commutation rule of $R$, it is expected that the Poisson structures of $R/hR$ are heavily related to the algebraic structures of deformation algebras of $R$ since they are induced by the same algebra $R$. In fact, algebraic structures of quantized algebras are  analogues to Poisson structures of their semiclassical limits as seen in many cases \cite{Good3}, \cite{Good4}, \cite{JoOh},  \cite{Oh12}, \cite{Oh7} and \cite{OhPa}.
 A main aim of this paper is to give a method how to construct Poisson algebras considered as a Poisson version of algebras related to a class of Ore extensions and an example illustrating this method. Namely, we observe \cite[Proposition 4.1]{LaLe} that Poisson polynomial extensions appear as semiclassical limits of a class of Ore extensions, construct a Poisson generalized Weyl algebra $A_1$ as an application and we verify that the Poisson endomorphisms of $A_1$ are Poisson analogues of the endomorphisms of quantum generalized Weyl algebra.

The Poisson polynomial extensions were constructed as a Poisson version of Ore extensions by the second author in \cite{Oh8}
and the quantum generalized Weyl algebra $A(a(h),q)$ over a Laurent polynomial ring in one variable was constructed by Bavula \cite{Ba} and the endomorphisms of $A(a(h),q)$ were completely classified by Kitchin and Launois \cite{KiLa} in the case when $a(h)$ is not invertible and $q$ is not a root of unity. Here we find a natural map $\Gamma$ from Ore extensions onto their semiclassical limits and then, as an application, we construct a Poisson generalized Weyl algebra $A_1$ induced from $A(a(h),q)$ by using $\Gamma$. Next
we find a necessary and sufficient condition such that  $A_1$ is Poisson simple and establish that the Poisson endomorphisms of $A_1$ are Poisson analogues of the endomorphisms of $A(a(h),q)$ by classifying  the Poisson endomorphisms of $A_1$.

Assume throughout the paper that all algebras have  unity and that the base field is the complex number field $\Bbb C$.


Let us begin with recalling the following basic terminologies.

\begin{defn}
(1) Let $\Bbb F$ be a commutative $\Bbb C$-algebra. Given  an $\Bbb F$-automorphism $\alpha$ on an $\Bbb F$-algebra $R$,  an $\Bbb F$-linear map  $\delta$ is said to be a {\it left $\alpha$-derivation} on $R$ if
$$\delta(ab)=\delta(a)b+\alpha(a)\delta(b)$$ for all $a,b\in R$.   For such  pair $(\alpha,\delta)$,  there exists a skew polynomial $\Bbb F$-algebra (or Ore extension) $R[z;\alpha,\delta]$.
Refer to \cite[Chapter 2]{GoWa2} for details of the skew polynomial ring.

(2) A commutative $\Bbb C$-algebra $A$ is said to be  a {\it Poisson algebra} if there exists a bilinear product $\{-,-\}$ on $A$, called a {\it Poisson bracket}, such that $(A, \{-,-\})$ is a Lie algebra and $\{ab,c\}=a\{b,c\}+\{a,c\}b$ for all $a,b,c\in A$.

We recall \cite[1.1]{Oh8}. A derivation $\alpha$ on $A$ is said to be a {\it Poisson derivation} if $$\alpha(\{a,b\})=\{\alpha(a),b\}+\{a,\alpha(b)\}$$ for all $a,b\in A$. Let   $\alpha$ be a Poisson derivation and let $\delta$ be a derivation on $A$. If the pair $(\alpha,\delta)$ satisfies the following condition
\begin{equation}\label{PSKE}
\delta(\{a,b\})-\{\delta(a),b\}-\{a,\delta(b)\}=\alpha(a)\delta(b)-\delta(a)\alpha(b)
\end{equation}
for all $a,b\in A$, then  the commutative polynomial algebra $A[z]$ becomes a Poisson  algebra  with Poisson bracket $$\{z,a\}=\alpha(a)z+\delta(a)$$ for all $a\in A$. Such Poisson algebra $A[z]$ is called a  Poisson polynomial extension (or Poisson Ore extension) and  denoted by $A[z;\alpha,\delta]_p$. (In \cite[1.1]{Oh8},  $\{z,a\}$ is defined by $\{z,a\}=-\alpha(a)z-\delta(a)$ for all $a\in A$.)  If $\alpha=0$  then
we write $A[z;\delta]_p$ for $A[z;0,\delta]_p$ and if  $\delta=0$ then we write $A[z;\alpha]_p$ for $A[z;\alpha,0]_p$.

(3) An ideal $I$ of a Poisson algebra $A$ is said to be a {\it Poisson ideal} if $\{I,A\}\subseteq I$.  A Poisson ideal $P$ is said to be {\it Poisson prime} if, for all Poisson ideals $I$ and $J$, $IJ\subseteq P$ implies $I\subseteq P$ or $J\subseteq P$.  If $A$ is noetherian  then a Poisson prime ideal of $A$ is a prime ideal by \cite[Lemma 1.1(d)]{Good3}.
\end{defn}


\section{Polynomial extensions}

Let $t$ be an indeterminate.

\begin{notn}\label{ASSUM}
Let  a 5-tuple $({\bf K},\Bbb F, A, t-1, (\alpha,\delta))$ satisfy the following conditions (1)-(5):

(1) Assume that  ${\bf K}$ is an infinite subset of the set $\Bbb C\setminus\{0,1\}$.

(2) Assume that
$\Bbb F$ is a subring of the ring of regular functions on ${\bf K}\cup\{1\}$ containing $\Bbb C[t,t^{-1}]$. That is,
\begin{equation}\label{REG}
\Bbb C[t,t^{-1}]\subseteq \Bbb F\subseteq\{f/g\in\Bbb C(t)| f,g\in\Bbb C[t]\text{ such that } g(1)\neq0,g(\lambda)\neq0\ \forall\lambda\in {\bf K}\}.
\end{equation}

(3) Assume that $A$ is an $\Bbb F$-algebra generated by $x_1,\ldots,x_n$. Note that $A$ is also  a $\Bbb C$-algebra since $\Bbb C\subseteq \Bbb F$.

(4) Assume that  $t-1$ is a nonzero, nonunit and non-zero-divisor of $A$ such that the factor $A_1=A/(t-1)A$ is commutative.
(Note that $t-1\in\Bbb F$ is a central element of $A$ and thus $(t-1)A$ is an ideal of $A$.)

(5) Assume that  $\alpha$ and $\delta$ are $\Bbb F$-linear maps from $A$ into itself such that $\alpha$ is an automorphism,  $\delta$ is a left $\alpha$-derivation and  the pair $(\alpha,\delta)$ satisfies the condition
\begin{equation}\label{AD}
(\alpha-\text{id})(A)\subseteq (t-1)A,\ \ \ \delta(A)\subseteq (t-1)A,
\end{equation}
where $\text{id}$ is the identity map on $A$.
\end{notn}

By Notation~\ref{ASSUM}(5),  there exists the skew polynomial $\Bbb F$-algebra $$B:=A[z;\alpha,\delta].$$
For each $\lambda\in{\bf K}\cup\{1\}$, $(t-\lambda)A$ and $(t-\lambda)B$ are ideals of $A$ and $B$ respectively since $t-\lambda$ is a central element of $A$ and $B$. Set
$$A_\lambda:=A/(t-\lambda)A,\ \ \ B_\lambda:=B/(t-\lambda)B.$$
For an element $a$ of $A$ or $B$, denote by $\overline{a}$ the canonical image of $a$ in $A_\lambda$ and $B_\lambda$.
For each $\lambda\in{\bf K}$, define
$$\begin{aligned}
\alpha_\lambda&:A_\lambda\longrightarrow A_\lambda,& \ \  \alpha_\lambda(\overline{a})&=\overline{\alpha(a)},\\
\delta_\lambda&:A_\lambda\longrightarrow A_\lambda,& \ \  \delta_\lambda(\overline{a})&=\overline{\delta(a)}.
\end{aligned}$$

\begin{lem}\label{ITER}
(1) For each $\lambda\in{\bf K}$, $\alpha_\lambda$ is an automorphism  and $\delta_\lambda$ is a left $\alpha_\lambda$-derivation.

(2) For each $\lambda\in{\bf K}$,  $B_\lambda\cong A_\lambda[z;\alpha_\lambda,\delta_\lambda]$ as $\Bbb C$-algebras.
\end{lem}

\begin{proof}
(1) Since $\alpha$ and $\delta$ are $\Bbb F$-linear maps and $t-\lambda\in\Bbb F$, $\alpha_\lambda$ and $\delta_\lambda$ are well-defined.
Since $\alpha$ is an automorphism on $A$, $\alpha_\lambda$ is an automorphism on $A_\lambda$. Moreover $\delta_\lambda$ is a left $\alpha_\lambda$-derivation because
$\delta$ is a left $\alpha$-derivation.

(2) Let $\psi:A\longrightarrow A_\lambda[z;\alpha_\lambda,\delta_\lambda]$ be the  map defined by $\psi(a)=\overline{a}$. Then,
by \cite[Proposition 2.4]{GoWa2}, there is an extension $\overline{\psi}$ of $\psi$ to $B$
such that $\overline{\psi}(z)=z$ since $z\overline{a} =\overline{\alpha(a)}z+\overline{\delta(a)}$ for all $a\in A$. Clearly $\ker\overline{\psi}=(t-\lambda)B$ and $\overline{\psi}$ is surjective.
\end{proof}

 Since $t-1$ is a nonzero, nonunit and non-zero-divisor of $A$, it is also a  nonzero, nonunit and non-zero-divisor of $B$.
 Moreover $B_1$ is a commutative $\Bbb C$-algebra by (\ref{AD}).
 Hence $A_1$ and $B_1$ are Poisson algebras with Poisson brackets
\begin{equation}\label{POR}
\{\overline{a},\overline{b}\}=\overline{(t-1)^{-1}(ab-ba)}
\end{equation}
for all $a,b\in A$ and $a,b\in B$ by \cite[III.5.4]{BrGo}. The $\Bbb C$-algebras $A_1$ and $B_1$ are  said to be  {\it semiclassical limits} of $A$ and $B$ respectively. For each $\lambda\in{\bf K}$, the $\Bbb C$-algebras $A_\lambda$ and $B_\lambda$ are  said to be  {\it deformations} of $A$ and $B$ respectively.

\begin{lem}\label{CENT}
If $b\in B$ is a central element then $\overline{b}\in B_1$ is a Poisson central element.
\end{lem}

\begin{proof}
It is clear by (\ref{POR}).
\end{proof}

In the following theorem, note that the maps $\alpha_1,\delta_1$ are constructed in a different way from $\alpha_\lambda,\delta_\lambda$ for $\lambda\in{\bf K}$.

\begin{thm}\label{SC} \cite[Proposition 4.1]{LaLe}
$B_1\cong A_1[z;\alpha_1,\delta_1]_p$ as Poisson $\Bbb C$-algebras, where $\alpha_1$ and $\delta_1$ are defined by
$$\begin{aligned}
\alpha_1( \overline{a})&=\overline{(t-1)^{-1}({\alpha(a)-a})},\\
 \delta_1(\overline{a})&=\overline{(t-1)^{-1}{\delta(a)}}\end{aligned}$$
for all $a\in A$.
\end{thm}

For each $\lambda\in{\bf K}$, let $\Bbb C_\lambda=\Bbb C$.
Note  that,  for $f(t)\in\Bbb F$, the complex number $f(\lambda)$ is well-defined by (\ref{REG}).
Define a $\Bbb C$-algebra homomorphism
 $$\gamma_{_\Bbb F}:\Bbb F\longrightarrow \prod_{\lambda\in{\bf K}}\Bbb C_\lambda,\ \ \ \gamma_{_\Bbb F}(f(t))=(f(\lambda))_{\lambda\in{\bf K}}.$$

\begin{lem}\label{FIR}
The $\Bbb C$-algebra homomorphism $\gamma_{_\Bbb F}$ is injective.
\end{lem}

\begin{proof}
Suppose that $\gamma_{_\Bbb F}(f(t))=0$. Then $f(\lambda)=0$ for all $\lambda\in{\bf K}$. Since ${\bf K}$ is an infinite set and every nonzero polynomial has only finite zeros, $f(t)=0$. Hence $\gamma_{_\Bbb F}$ is injective.
\end{proof}

 Set
$$\widehat{A}=\prod_{\lambda\in{\bf K}}A_\lambda,\ \ \ \widehat{B}=\prod_{\lambda\in{\bf K}}B_\lambda.$$
Let  $$\pi_\lambda:\widehat{A}\longrightarrow A_\lambda,\ \ \ \pi_\lambda:\widehat{B}\longrightarrow B_\lambda$$ be the canonical projections onto $A_\lambda$ and $B_\lambda$ for each $\lambda\in{\bf K}$
 and let $\gamma_A$ and $\gamma_B$ be the $\Bbb C$-algebra homomorphisms
\begin{equation}\label{gamm}
\begin{aligned}
\gamma_A&:A\longrightarrow \widehat{A},&\ \ \gamma(a)&=(\overline{a})_{\lambda\in{\bf K}},\\
\gamma_B&:B\longrightarrow \widehat{B},&\ \ \gamma(b)&=(\overline{b})_{\lambda\in{\bf K}}.
\end{aligned}
\end{equation}
 Thus $\pi_\lambda\gamma_A(a)=\overline{a}$ and $\pi_\lambda\gamma_B(b)=\overline{b}$ for each $\lambda\in{\bf K}$.

\begin{lem}\label{INJ}
$\gamma_A$ is injective if and only if $\gamma_B$ is injective.
\end{lem}

\begin{proof}
If $\gamma_B$ is injective then $\gamma_A$ is also injective since $\gamma_A$ is the restriction of $\gamma_B$ to $A$.
Suppose that $\gamma_A$ is injective and that $\gamma_B(\sum_i a_iz^i)=0$, where $a_i\in A$. Then $\overline{a_i}=0$ for each $\lambda\in{\bf K}$. It follows $\gamma_A(a_i)=0$.
Hence $a_i=0$ for all $i$ since $\gamma_A(a_i)=0$ and $\gamma_A$ is injective.
\end{proof}

\begin{thm}\label{GAMMA}
Suppose that $\gamma_A$ is injective. Set
$$\begin{aligned}
&\Gamma_A:\gamma_A(A)\longrightarrow A_1,&\ \ \ \Gamma_A(x)&=\overline{\gamma_A^{-1}(x)},\\
&\Gamma_B:\gamma_B(B)\longrightarrow B_1,&\ \ \ \Gamma_B(x)&=\overline{\gamma_B^{-1}(x)}.
\end{aligned}$$
Then the $\Bbb C$-algebra homomorphisms $\Gamma_A$ and $\Gamma_B$ are surjective.
\end{thm}

\begin{proof}
By Lemma~\ref{INJ}, $\gamma_A$ and $\gamma_B$ are injective and thus $\Gamma_A$ and $\Gamma_B$ are well-defined $\Bbb C$-algebra homomorphisms.
Since the canonical maps from $A$ into $A_1$ is surjective, $\Gamma_A$ is surjective clearly. Similarly $\Gamma_B$ is  surjective.
\end{proof}


\section{Poisson generalized Weyl algebra}

The following quantum generalized Weyl algebra $A(a(h),q)$ is a special case of the generalized Weyl algebra introduced by Bavula in \cite{Ba}.

\begin{defn}
Let $0\neq q\in\Bbb C$ be not a root of unity and let
$0\neq a(h)\in\Bbb C[h^{\pm1}]$.
The quantum generalized Weyl algebra $A(a(h),q)$  is  the $\Bbb C$-algebra generated by $h^{\pm1}, x,y$ subject to the relations
$$xh=qhx,\ \ yh=q^{-1}hy,\ \ xy=a(qh),\ \ yx=a(h),\ \ h^{\pm1} h^{\mp1}=1.$$
\end{defn}

Set
$$\Bbb F=\Bbb C[t,t^{-1}],\ \ {\bf K}=\Bbb C\setminus(\{0,1\}\cup\{\text{roots of unity}\}).$$

Assume throughout the section that $0\neq q\in\Bbb C$ is {\bf not a root of unity} and that
$0\neq a(h)\in\Bbb C[h^{\pm1}]$ is {\bf not invertible}. Hence $q\in{\bf K}$ and $a(h)$ has at least two nonzero terms.

Set
$$B:=\Bbb F[h^{\pm1}][x;\alpha][y;\beta,\delta],$$ where
\begin{equation}\label{RELB}
\begin{aligned}
\alpha(h)&=th,&&\\
\beta(h)&=t^{-1}h,&\beta(x)&=x,\\
\delta(h)&=0,&\delta(x)&=a(h)-a(th).
\end{aligned}
\end{equation}

\begin{lem}\label{CENTB}
The element $xy-a(th)\in B$ is a central element.
\end{lem}

\begin{proof}
It is proved routinely by (\ref{RELB}).
\end{proof}

 Denote by $B(a(h),q)$ the $\Bbb C$-algebra generated by
$h^{\pm1}, x,y$ subject to the relations
\begin{equation}\label{BQREL}
xh=qhx,\ \ yh=q^{-1}hy,\ \ yx=xy+a(h)-a(qh),\ \ h^{\pm1} h^{\mp1}=1,
\end{equation}
which is obtained from $B$ by substituting $q$ for $t$. The $\Bbb C$-algebra $B(a(h),q)$
is an iterated skew polynomial algebra
$$B(a(h),q)=\Bbb C[h^{\pm1}][x;\alpha_q][y;\beta_q,\delta_q],$$
where $\alpha_q,\beta_q,\delta_q$ are the maps induced by $\alpha,\beta,\delta$ respectively.
Moreover
\begin{equation}\label{GENE}
B(a(h),q)\cong B/(t-q)B=B_q
\end{equation}
as $\Bbb C$-algebras by Lemma~\ref{ITER}(2).
For each $\lambda\in{\bf K}$, $\lambda$ is not a root of unity and thus there exists the $\Bbb C$-algebra $B(a(h),\lambda)\cong B_\lambda$ which is obtained from (\ref{GENE}) by
substituting $\lambda$ for $q$. Observe that  $q$ is not only a nonzero and non-root of unity but also plays a role as a parameter taking values in ${\bf K}$.

\begin{obsn}
In $B(a(h),q)$, $q$ plays a role as a parameter taking values in ${\bf K}$ and thus, for each $\lambda\in{\bf K}$, there exists an evaluation map
\begin{equation}\label{EVAL}
e_\lambda:B(a(h),q)\longrightarrow B_\lambda,\ \  \ f(q)\mapsto f(\lambda).
\end{equation}
\end{obsn}

Note that  the 5-tuples $({\bf K},\Bbb F, \Bbb F[h^{\pm1}], t-1, (\alpha,0))$ and $({\bf K},\Bbb F,\Bbb F[h^{\pm1}][x;\alpha], t-1, (\beta,\delta))$ satisfy  Notation~\ref{ASSUM}(1)-(5).
Applying Theorem~\ref{SC} to $\Bbb F[h^{\pm1}][x;\alpha]$ and $B$, there exists the Poisson $\Bbb C$-algebra
$$B_1=\Bbb C[h^{\pm1}][x;\alpha_1]_p[y;\beta_1,\delta_1]_p,$$ where
\begin{equation}\label{BRAC}
\begin{aligned}
\alpha_1(h)&=h,&&\\
\beta_1(h)&=-h,&\beta_1(x)&=0,\\
\delta_1(h)&=0,&\delta_1(x)&=-a'(h)h,
\end{aligned}
\end{equation}
here $a'(h)$ is the formal derivative of $a(h)$.

Define a $\Bbb C$-algebra homomorphism $\gamma_B$   by
$$\gamma_B:B\longrightarrow \widehat{B}=\prod_{\lambda\in{\bf K}}B_\lambda,\ \ \ \gamma_B(z)=(\overline{z})_{\lambda\in{\bf K}}.$$
Applying Lemma~\ref{INJ} to $\gamma_B$ inductively, $\gamma_B$ is injective by  Lemma~\ref{FIR} and thus, by Theorem~\ref{GAMMA},  there exists the $\Bbb C$-algebra homomorphism
 $$\Gamma_B:\gamma_B(B)\longrightarrow B_1,\ \ \ \Gamma_B=\gamma_1\gamma_B^{-1},$$
 where $\gamma_1:B\longrightarrow B_1=B/(t-1)B$ is the canonical projection.

Define a map (not $\Bbb C$-linear map)
$$\widehat{}\ :B(a(h),q)\longrightarrow \widehat{B}=\prod_{\lambda\in{\bf K}}B_\lambda,\ \ \  f(q)\mapsto \widehat{f(q)}:=(f(\lambda))_{\lambda\in{\bf K}}$$
which exists by (\ref{EVAL}).
Note that the image of \ $\widehat{}$ \ is equal to the image of $\gamma_B$.
Hence there exists  the composition $\Gamma:=\Gamma_B\circ\widehat{}\ \ .$
Roughly speaking, $\Gamma$ is a map defined by plugging 1 to $q$.
\begin{equation}\label{GAMM}
\Gamma:B(a(h),q){\longrightarrow} \widehat{B}=\prod_{\lambda\in{\bf K}}B_\lambda\longleftarrow B\longrightarrow B_1.
\end{equation}
Note that $\Gamma$ is surjective.
By Lemma~\ref{CENT} and Lemma~\ref{CENTB}, $\Gamma(xy-a(qh))=xy-a(h)$ is a Poisson central element of $B_1$.
Thus $(xy-a(h))B_1$ is a Poisson ideal.
Set
$$A_1:=B_1/(xy-a(h))B_1.$$

Henceforth, we simply write $w$ for $\overline{w}\in A_1$.

\begin{thm}
(1) The quantum generalized Weyl algebra $A(a(h),q)$ is equal to the factored $\Bbb C$-algebra $B(a(h),q)/(xy-a(qh))B(a(h),q)$.

(2) The Poisson algebra $A_1$ is induced from $A(a(h),q)$ by $\Gamma$ in (\ref{GAMM}).

(3) The Poisson  algebra $A_1$  is the $\Bbb C$-algebra $\Bbb C[h^{\pm1},x,y]/\langle xy-a(h)\rangle$ with Poisson bracket
\begin{equation}\label{OH}
\{x,h\}=hx,\ \ \{y,h\}=-hy,\ \ \{y,x\}=-a'(h)h.
\end{equation}

We will call $A_1$ a {\it Poisson generalized  Weyl algebra}.
\end{thm}

\begin{proof}
(1) It follows by (\ref{BQREL}).

(2) It follows by the facts that $\Gamma(B(a(h),q))=B_1$ and $\Gamma(xy-a(qh))=xy-a(h)$.

(3) Since $B_1=\Bbb C[h^{\pm1},x,y]$, the result follows by (\ref{BRAC}).
\end{proof}

Let $R$ be a Poisson algebra. If there exists subspaces $R_k$, $k\in\Bbb Z$, such that $R=\bigoplus_{k\in\Bbb Z}R_k$ and
$$R_kR_\ell\subseteq R_{k+\ell},\ \ \{R_k,R_\ell\}\subseteq R_{k+\ell}$$ for all $k,\ell\in\Bbb Z$ then $R$ is said to be a {\it $\Bbb Z$-graded Poisson algebra}.

Give degrees on the generators $h,x,y$ of $B_1$ by $\text{deg}(h)=0$, $\text{deg}(x)=1$ and $\text{deg}(y)=-1$. Then
$B_1$ is a $\Bbb Z$-graded Poisson algebra. Moreover the  Poisson ideal $\langle xy-a(h)\rangle$ is graded. Thus $A_1$ is also a $\Bbb Z$-graded Poisson algebra
\begin{equation}\label{GRA}
A_1=\bigoplus_{k\in\Bbb Z}W_k,
\end{equation}
where
$$W_k=\left\{\begin{aligned} &\Bbb C[h^{\pm1}]x^k&&\text{if }k>0,\\ &\Bbb C[h^{\pm1}]&&\text{if }k=0,\\  &\Bbb C[h^{\pm1}]y^{-k}&&\text{if }k<0.\end{aligned}\right.$$

\begin{lem}\label{EIG}
Define a $\Bbb C$-linear map $f$ by
 $$f:A_1\longrightarrow A_1,\ \ \  f(a)=\{a,h\}h^{-1}.$$
Then, for each $k\in\Bbb Z$, every nonzero element of $W_k$ is an eigenvector of $f$ with eigenvalue $k$.
\end{lem}

Note that $\Bbb C[h^{\pm1}]$ is a unique factorization domain since it is a principal ideal domain.

\begin{thm}\label{SIM}
The Poisson algebra $A_1$ is Poisson simple if and only if every root of $a(h)$ is a simple root.
\end{thm}

\begin{proof}
Let $b(h)\in\Bbb C[h^{\pm1}]$ be the greatest common divisor of $a(h)$ and $a'(h)$. Note that $a(h)$ has a  root with multiplicity $>1$ if and only if $b(h)$ is a nonunit.

$(\Rightarrow)$ Suppose that $a(h)$ has a  root with multiplicity $>1$. Thus $b(h)$ is a  nonunit. Let $I$ be the ideal of $A_1$ generated by $x,y$ and $b(h)$. Then $A_1/I$ is isomorphic to the algebra $\Bbb C[h^{\pm1}, x,y]/J$, where $J$ is the ideal of $\Bbb C[h^{\pm1}, x,y]$ generated by $x,y$ and $b(h)$, since $a(h)$ is divided by $b(h)$. Moreover $\Bbb C[h^{\pm1}, x,y]/J$ is isomorphic to $\Bbb C[h^{\pm1}]/b(h)\Bbb C[h^{\pm1}]$. Thus $I$ is a nontrivial ideal of $A_1$. Observe that
$$\begin{aligned}
\{x, b(h)\}&=b'(h)\{x,h\}=b'(h)a'(h)hx\in I,\\
 \{y, b(h)\}&=b'(h)\{y,h\}=-b'(h)a'(h)hy\in I\\
 \{y,x\}&=-a'(h)h\in I\end{aligned}$$
by the chain rule and (\ref{OH}). Hence $I$ is a nontrivial Poisson ideal of $A_1$ and thus $A_1$ is not Poisson simple.

$(\Leftarrow)$  Suppose that $A_1$ is not Poisson simple. Then there exists a nontrivial Poisson ideal $I$ of $A_1$. Let $P$ be a minimal prime ideal over $I$. By \cite[6.2]{Good4}, $P$ is a prime Poisson ideal.
Applying $f$ to $P$,  $P$ contains a nonzero element $w_k\in W_k$ for some $k\in\Bbb Z$ by Lemma~\ref{EIG}. If $k\geq0$ then $x\in P$ or $c(h)\in P$ for some $0\neq c(h)\in\Bbb C[h^{\pm1}]$
since $P$ is prime.
If $x\in P$ then $a'(h)h=\{x,y\}\in P$ and thus $a(h)\in P$ and $a'(h)\in P$. It follows that the greatest common divisor $b(h)$ of $a(h)$ and $a'(h)$ is an element of $P$ and thus $P$ contains a unit, a contradiction.
Similarly, if $k<0$ then repeating the argument gives that $y\notin P$ and that $P$ contains  an element $0\neq c(h)\in\Bbb C[h^{\pm1}]$.

Let  $0\neq c(h)\in P$ and $x,y\notin P$.
We may assume  $c(h)\in\Bbb C[h]$ by multiplying $h^n$ to $c(h)$ for sufficiently large $n$.
If the degree of $c(h)$ is greater than 1 then $0\neq c'(h)\in P$ since $$P\ni\{x,c(h)\}=c'(h)hx.$$ Repeating this process, we get that $P$ contains a unit, a contradiction. Hence $A_1$ is Poisson simple.
\end{proof}

Let us find the Poisson endomorphisms of $A_1$. The following arguments are  Poisson analogues of those in \cite{KiLa}. For completion, we repeat them.
Note that the unit group of $A_1$ is $\{\gamma h^i|\gamma\in\Bbb C^*,i\in\Bbb Z\}$.
Let $\psi$ be a Poisson endomorphism of $A_1$. Then $\psi(h)=\gamma h^i$ for some $\gamma\in\Bbb C^*$ and $i\in \Bbb Z$ since  Poisson endomorphisms preserve units.
Hence there are possible three types of Poisson endomorphisms of $A_1$ as in \cite{KiLa}. Positive-type Poisson endomorphisms, that is, Poisson endomorphisms $\psi$ such that
$\psi(h)=\gamma h^i$ $(i>0)$; Zero-type Poisson endomorphisms, that is, Poisson endomorphisms $\psi$ such that
$\psi(h)=\gamma$; Negative-type Poisson endomorphisms, that is, Poisson endomorphisms $\psi$ such that
$\psi(h)=\gamma h^i$ $(i<0)$.

In the following theorem, we see that  the positive-type Poisson endomorphisms and the negative-type Poisson endomorphisms of $A_1$ are the same forms as those of endomorphisms of $A(a(h),q)$ but the zero-type Poisson endomorphisms of $A_1$ are slightly different from those of $A(a(h),q)$.
(One should compare the following theorem with \cite[Proposition 3.1, Proposition 4.1 and Proposition 5.3]{KiLa}.)

\begin{thm}
Let $d$ be the maximal degree of $a(h)$.  Write $a(h)$ by
$$a(h)=a_{i_1}h^{i_1}+a_{i_2}h^{i_2}+\ldots+a_{i_m}h^{i_m},$$
where $i_1<i_2<\ldots<i_m=d$ and $a_{i_j}\in\Bbb C^*$ for all $j=1,2,\ldots, m$.
Denote by $k$  the greatest common divisor of $d-i_1, d-i_2,\ldots, d-i_{m-1}$.

(1) Let $\psi$ be a positive-type Poisson endomorphism of $A_1$. Then
\begin{equation}\label{PENDO}
\psi(h)=\gamma h, \ \ \ \psi(x)=bh^nx,\ \ \ \psi(y)=\gamma^db^{-1}h^{-n}y,
\end{equation}
where $\gamma$ is a $k$-th root of unity, $b\in\Bbb C^*$ and $n\in\Bbb Z$. Conversely, a map $\psi$ satisfying (\ref{PENDO}) determines a unique  Poisson automorphism of $A_1$.

(2) Let $\psi$ be a zero-type Poisson endomorphism of $A_1$. Then
\begin{equation}\label{ZENDO}
\psi(h)=\gamma, \ \ \ \psi(x)=0,\ \ \ \psi(y)=0,
\end{equation}
where $\gamma\in\Bbb C^*$ is a root of $a(h)$ with multiplicity $>1$. If $a(h)$ has a root with multiplicity $>1$ then  a map $\psi$ satisfying (\ref{ZENDO}) determines a unique Poisson endomorphism. If
every root of $a(h)$ has multiplicity 1 then there are no zero-type Poisson endomorphisms.

(3) Let $\psi$ be a negative-type Poisson endomorphism of $A_1$. Then
\begin{equation}\label{NENDO}
\psi(h)=\gamma h^{-1}, \ \ \ \psi(x)=ch^vy,\ \ \ \psi(y)=bh^{u}x,
\end{equation}
where $\gamma,b,c\in\Bbb C^*$ and $u,v\in\Bbb Z$ satisfy the relation
\begin{equation}\label{AENDO}
bch^{u+v}a(h)=a(\gamma h^{-1}).
\end{equation}
Conversely, a map $\psi$ satisfying (\ref{NENDO}) determines a unique  Poisson automorphism of $A_1$.
\end{thm}

\begin{proof}
Note that every Poisson endomorphism $\psi$ of $A_1$ preserves the following equations
\begin{equation}\label{RELBB}
\{x,h\}=hx,\ \ \{y,h\}=-hy,\ \ \{y,x\}=-a'(h)h,\ \ xy=a(h).
\end{equation}

(1) Let $\psi$ be a Poisson endomorphism such that $\psi(h)=\gamma h^j$ for $\gamma\in\Bbb C^*$ and $j>0$. Suppose that $\psi(x)=0$. Applying $\psi$ to the third equation of (\ref{RELBB}), the left hand side is zero and the right hand side is
 $-\gamma a'(\gamma h^j)h^j$ which is nonzero, a contradiction. Hence $\psi(x)\neq 0$. Similarly $\psi(y)\neq 0$.
 We can set $\psi(x)=\sum_{k\in\Bbb Z} w_k$ by (\ref{GRA}), where $w_k\in W_k$ for each $k$. Applying $\psi$ to $\{x,h\}=hx$, we have
$$\{\psi(x), \gamma h^j\}=\gamma h^j\psi(x).$$
The left hand side of the above equation is $\sum_k jk\gamma h^jw_k$ and thus $jk=1$ for all $k$ such that $w_k\neq0$. Thus we have that
$$\psi(h)=\gamma h,\ \ \ \psi(x)=b(h)x$$
for some $0\neq b(h)\in\Bbb C[h^{\pm1}]$. Repeating this argument on $\{y,h\}=-hy$, we get
$$\psi(y)=c(h)y$$
for some $0\neq c(h)\in\Bbb C[h^{\pm1}]$.

 Applying $\psi$ to the last equation of (\ref{RELBB}), we get $b(h)c(h)a(h)=a(\gamma h)$.
Comparing the maximal and the  minimal degrees for $h$ on both sides and then coefficients, we get
$$\psi(x)=bh^nx,\ \ \ \psi(y)=\gamma^db^{-1}h^{-n}y,$$
where $\gamma$ is a $k$-th root of unity, $b\in\Bbb C^*$ and $n\in\Bbb Z$.

Conversely, let $\psi$ be a map satisfying (\ref{PENDO}). Then $\psi$ determines a unique algebra endomorphism since $B_1$ is the commutative polynomial ring $\Bbb C[h^{\pm1},x,y]$ and  it preserves the last equation of (\ref{RELBB}).
 It is checked routinely that $\psi$ preserves the other equations of (\ref{RELBB}). Thus $\psi$ is a Poisson endomorphism. Such $\psi$ is a Poisson automorphism since there exists $\psi^{-1}$ defined by
$$\psi^{-1}(h)=\gamma^{-1}h,\ \ \psi^{-1}(x)=\gamma^{n}b^{-1}h^{-n}x,\ \
\psi^{-1}(y)=\gamma^{-n-d}bh^{n}y.
$$

(2) Let $\psi$ be a Poisson endomorphism such that $\psi(h)=\gamma$, where $\gamma\in\Bbb C^*$. Then $\psi(x)=\psi(y)=0$ by the first and the second equations of
(\ref{RELBB}). Applying $\psi$ to the third and the last equations of (\ref{RELBB}), we get that $\gamma$ is a common root of $a(h)$ and $a'(h)$. Thus $\gamma$ is a root of $a(h)$ with multiplicity $>1$.

Conversely, let $\psi$ be a map satisfying (\ref{ZENDO}). Then $\psi$ determines a unique algebra endomorphism since $\psi$ preserves the last
equation of (\ref{RELBB}). Moreover $\psi$ preserves the other equations of (\ref{RELBB}) and thus $\psi$ is a Poisson endomorphism. Now the other statements are trivial by Theorem~\ref{SIM}.

(3) Let $\psi$ be a Poisson endomorphism such that $\psi(h)=\gamma h^j$ for $\gamma\in\Bbb C^*$ and $j<0$. Since $\psi^2$ is of positive type, $\psi(x)\neq0$ and $\psi(y)\neq0$ by (1).
 We can set $\psi(x)=\sum_{k\in\Bbb Z} w_k$ by (\ref{GRA}), where $w_k\in W_k$ for each $k$. Applying $\psi$ to $\{x,h\}=hx$, we have
$$\{\psi(x), \gamma h^j\}=\gamma h^j\psi(x).$$
The left hand side of the above equation is $\sum_k jk\gamma h^jw_k$ and thus $jk=1$ for all $k$ such that $w_k\neq0$. Thus we have that
$$\psi(h)=\gamma h^{-1},\ \ \ \psi(x)=c(h)y$$
for some $0\neq c(h)\in\Bbb C[h^{\pm1}]$. Repeating this argument on $\{y,h\}=-hy$, we get
$$\psi(y)=b(h)x$$
for some $0\neq b(h)\in\Bbb C[h^{\pm1}]$. Moreover we have that $b(h)= bh^u, c(h)=ch^v$ for some $b,c\in\Bbb C^*$ and $u,v\in\Bbb Z$ by (1) since $\psi^2$ is of positive type.
 Applying $\psi$ to the last equation of (\ref{RELBB}), we get the relation $bch^{u+v}a(h)=a(\gamma h^{-1})$.

Conversely, let $\psi$ be a map satisfying (\ref{NENDO}). Then $\psi$ preserves the last equation of (\ref{RELBB}) and thus $\psi$ determines a unique algebra endomorphism. It is easy to check that $\psi$ preserves the first and the second equations of (\ref{RELBB}).
Moreover it is shown that $\psi$ preserves the third equation of (\ref{RELBB}) by differentiating the equation (\ref{AENDO}) by $h$.
Hence  $\psi$  satisfying (\ref{NENDO}) determines a unique Poisson endomorphism. Since there exists a map $\psi^{-1}$ of the negative type, $\psi$ determines a unique Poisson automorphism.
\end{proof}

\noindent
{\bf Acknowledgments}  The authors would like to appreciate S. Launois for informing a reference \cite[Proposition 4.1]{LaLe} about semiclassical limits of Ore extensions, that is used for constructing a Poisson generalized Weyl algebra, after the first version of the paper is posted.

 The second author is supported by Chungnam Nationality University Grant and thanks the Korea Institute for Advanced Study for the warm hospitality during a part of the preparation of this paper.


\bibliographystyle{amsplain}


\providecommand{\bysame}{\leavevmode\hbox to3em{\hrulefill}\thinspace}
\providecommand{\MR}{\relax\ifhmode\unskip\space\fi MR }
\providecommand{\MRhref}[2]{%
  \href{http://www.ams.org/mathscinet-getitem?mr=#1}{#2}
}
\providecommand{\href}[2]{#2}

\end{document}